\documentclass[11pt]{amsart}
\usepackage{amssymb}
\usepackage{amsmath}
\usepackage{amsfonts}
\usepackage{graphicx}

\usepackage[total={17cm,22cm},top=2.5cm, left=2.3cm]{geometry}
\usepackage{hyperref}
    \usepackage{aeguill}
    \usepackage{type1cm}

\theoremstyle{plain}
\newtheorem{thm}{Theorem}[section]

\newtheorem{claim}[thm]{Claim}

\newtheorem{example}[thm]{Example}

\newtheorem{remark}[thm]{Remark}

\newtheorem{theorem}[thm]{Theorem}
\newtheorem{fact}[thm]{Fact}

\newtheorem{ex}[thm]{Example}

\numberwithin{equation}{section}
\newcommand{\N}{\mathbb{N}}

\newcommand{\R}{\mathbb{R}}

\DeclareMathOperator{\dom}{dom\,\!}

\DeclareMathOperator{\dist}{dist\,\!}

\DeclareMathOperator{\interior}{int\,\!}

\usepackage[usenames,dvipsnames]{color}

\usepackage[dvipsnames]{xcolor}

\begin{document}

\title[Extensions of convex functions with prescribed subdifferentials]{Extensions of convex functions with prescribed subdifferentials}

\author{Daniel Azagra}
\address{ICMAT (CSIC-UAM-UC3-UCM), Departamento de An{\'a}lisis Matem{\'a}tico y Matem\'atica Aplicada,
Facultad Ciencias Matem{\'a}ticas, Universidad Complutense, 28040, Madrid, Spain.  {\bf DISCLAIMER:} 
The first-named author is affiliated with Universidad Complutense de Madrid, but this does not mean this institution has offered him all the support he expected. On the contrary, the Biblioteca Complutense has hampered his research by restricting his access to books.
}
\email{azagra@mat.ucm.es}

\author{Juan Ferrera}
\address{IMI, Departamento de An{\'a}lisis Matem{\'a}tico y Matem{\'a}tica Aplicada,
Facultad Ciencias Matem{\'a}ticas, Universidad Complutense, 28040, Madrid, Spain }
\email{ferrera@mat.ucm.es}

\author{Javier G\'{o}mez-Gil}
\address{Departamento de An{\'a}lisis Matem{\'a}tico y Matem{\'a}tica Aplicada,
Facultad Ciencias Matem{\'a}ticas, Universidad Complutense, 28040, Madrid, Spain }
\email{gomezgil@mat.ucm.es}

\author{Carlos Mudarra}
\address{ICMAT (CSIC-UAM-UC3-UCM), Calle Nicol\'as Cabrera 13-15.
28049 Madrid, Spain}
\email{carlos.mudarra@icmat.es}

\date{November 10, 2018}



\begin{abstract}
Let $E$ be an arbitrary subset of a Banach space $X$, $f: E \rightarrow \mathbb{R}$ be a function, and $G:E \rightrightarrows X^*$ be a set-valued mapping. We give necessary and sufficient conditions on $f, G$ for the existence of a continuous convex extension $F: X \rightarrow \mathbb{R} $ of $f$ such that the subdifferential $\partial F$ of $F$ coincides with $G$ on $E.$
\end{abstract}

\maketitle

\section{Introduction and main results}

In this paper we are concerned with the following nonsmooth extension problem.

\medskip 

{\em Let $X$ be a Banach space and $E$ be an arbitrary subset of $X$. If $f:E \to \R$ is a function and $G:E \rightrightarrows X^*$ is a set-valued mapping, what necessary and sufficient conditions on $f, G$ will guarantee the existence of a continuous convex extension $F$ of $f$ to all of $X$ such that $\partial F(x)=G(x)$ for every $x\in E$? }

\medskip

As a consequence of the results of K. Schulz and B. Schwartz in \cite{SchulzSchwartz}, a partial answer to this question was given in the finite-dimensional setting; more precisely, if $E$ is convex and $G$ is the subdifferential mapping of $f$ on $E,$ then there exists a finite convex extension $F$ of $f$ to all of $\R^n$ with $G(x) \subset \partial F(x)$ for every $x\in E$ if and only if, for every pair of sequences $(y_k)_k \subset E, \: (y_k^*)_k$ with $y_k^* \in G(y_k)$ and $\lim_k |y_k^*|= +\infty,$ one has that $\lim_k \frac{\langle y_k,y_k^*\rangle-f(y_k)}{|y_k^*|}= +\infty$ too. Moreover, whenever this condition is satisfied, the function
$$
F(x)=\sup \lbrace f(y)+ \langle y^*, x-y \rangle \: : \: y\in E, \: y^*\in G(y) \rbrace,\quad x\in \R^n;
$$ defines such an extension. Nevertheless the subdifferential $\partial F(x)$ of $F$ at points $x\in E$ does not necessarily coincide with $G(x).$

\medskip

Very recently, a related problem has been solved for convex functions of the classes $C^1(\R^n)$ and $C^{1, \omega}(X)$ (for a Hilbert space $X$) in the situation where the mapping $G$ is single-valued and one additionally requires that the extension $F$ be of class $C^1(\R^n)$ (which amounts to asking that $\partial F(x)$ be a singleton for every $x\in \R^n$) or of class $C^{1, \omega}(X)$; see \cite{AzagraMudarraExplicitFormulas, AzagraMudarra2017PLMS, AzagraMudarraGlobalGeometry}. A solution to a similar problem for general (not necessarily convex) functions was given in \cite[Theorem 5]{FerreraGil}, characterizing the pairs $f:E\to \R$, $G:E \rightrightarrows \R^n$ with $f$ continuous and $G$ upper semicontinuous and nonempty, compact and convex-valued which admit a (generally nonconvex) extension $F$ of $f$ whose Fr\'echet subdifferential is upper semicontinuous on $\R^n$ and extends $G$ from $E.$ 

\medskip

Let us also mention the results by B. Mulansky and M. Neamtu \cite{Mulansky} which prove that any finite subset of data on $\R$ or $\R^2$ which is {\em strictly convex} in an appropriate sense can be interpolated by a convex polynomial and the work by O. Bucicovschi and J. Lebl \cite{BucicovschiLebl} which study the continuity and regularity of extensions of functions defined on compact subset $K$ of $\R^n$ to the convex hull of $K.$

\medskip

In the infinite-dimensional setting, several characterizations of those pairs of Banach spaces $Y \subset X$ for which every continuous convex function defined on $Y$ admits a continuous convex extension to the superspace $X$ have been established by J. Borwein, S. Fitzpatrick, V. Montesinos and J. Vanderwerff in \cite{BorFitzVan, BorweinMontesinosVanderwerff} and by C. A. De Bernardi and L. Vesel\'y in \cite{BernardiVesely1, BernardiVesely2, Bernardi}.  A consequence of these results is that, if $X$ is a normed space and $Y$ is a closed subspace of $X$ such that $X/Y$ is separable, then every continuous convex function on $Y$ can be extended to a continuous convex function on $X.$ On the other hand, if $X=\ell_\infty$ and $Y=c_0$ or $\ell_p$ with $1<p<\infty,$ there are examples of continuous convex functions on $Y$ which have no continuous convex extensions to all of $X.$ Also \cite[Theorem 3.1]{BernardiVesely1} asserts that if $A$ is an open convex subset of a topological vector space $X$, $Y$ is a subspace of $X$ and $f: A \cap Y \to \R$ is a continuous convex function, then $f$ admits a continuous convex extension $F: A \to \R$ if and only if $A= \bigcup_n A_n$  for some increasing sequence $\lbrace A_n \rbrace_n$ of open convex subsets of $A$ such that $f$ is bounded on each $A_n \cap Y.$ See also the paper \cite{VeselyZajicek} by L. Vesel\'y and L. Zaj\'i\v{c}ek for results on the extension of delta-convex functions.

\medskip

The following theorem is one of the main results of this paper and provides a complete answer to the problem in finite-dimensional spaces. 

\begin{theorem}\label{maintheoremfinitedimensional}
Let $E \subset \R^n$ be arbitrary, $f: E \to \R$ be a function and $G: E \rightrightarrows \R^n$ be a set-valued mapping such that $G(x)$ is compact, convex and nonempty for every $x\in E.$ There exists a convex function $F: \R^n \to \R$ such that $F(x)=f(x)$ and $\partial F(x)= G(x)$ for every $x\in E$ if and only if the following conditions are satisfied.
\item[] $(C):$ $f(x) \geq f(y)+ \langle y^*,x-y \rangle$ for every $x,y\in E, \:y^*\in G(y).$ 
\item[] $(EX):$ If $(y_k)_k \subset E, \: (y_k^*)_k \subset \R^n$ are such that $y_k^* \in G(y_k)$ for every $k$ and $\lim_k |y_k^*| =+\infty,$ then
$$
\lim_k \frac{\langle y_k,y_k^*\rangle-f(y_k)}{|y_k^*|}= +\infty.
$$
\item[] $(SCS):$ If $x\in E$ and $(y_k)_k \subset E, \: (y_k^*)_k \subset \R^n$ are such that $y_k^* \in G(y_k)$ for every $k,$ $(y_k^*)_k$ is bounded, and 
$$
\lim_k \left( f(x)-f(y_k)-\langle y_k^*, x-y_k \rangle \right)=0,
$$
then there exists a sequence $(x_k^*)_k \subset G(x)$ such that $\lim_k |x_k^*-y_k^*|=0.$
\end{theorem}

As we mentioned above, condition $(EX)$ was considered in \cite{SchulzSchwartz} in order to ensure the existence of convex extensions $F$ of $f$ such that $G \subseteq \partial F$ on $E.$ It is condition $(SCS)$ which allows us to obtain the exact identity $\partial F= G$ on $E.$ If $E$ is compact, then condition $(EX)$ trivially holds and, assuming that $G$ is upper semicontinuous (which we can fairly do since the subdifferential of a continuous convex function on $\R^n$ is upper semicontinuous), condition $(SCS)$ can be simplified by replacing sequences with points. Recall that if $Y,Z$ are two topological spaces, a set-valued mapping $G:Y \rightrightarrows Z$ is said to be upper semicontinuous on $Y$ provided that for every $y\in Y$ and every open set $W$ in $Z$ containing $G(y),$ there exists a neighbourhood $V$ of $y$ such that $G( V) \subset W$. For any undefined terms in Convex Analysis that we may use in this paper, we refer to the books \cite{BorweinVanderwerffbook, Rockafellar, Zalinescubook}.

\begin{theorem}\label{maintheoremfinitedimensionalcompact}
Let $E \subset \R^n$ be compact, $f: E \to \R$ be a function and $G: E \rightrightarrows \R^n$ be an upper semicontinuous set-valued mapping such that $G(x)$ is compact, convex and nonempty for every $x\in E.$ There exists a convex function $F: \R^n \to \R$ such that $F(x)=f(x)$ and $\partial F(x)= G(x)$ for every $x\in E$ if and only if the following conditions are satisfied.
\item[] $(C):$ $f(x) \geq f(y)+ \langle y^*,x-y \rangle$ for every $x,y\in E, \:y^*\in G(y).$ 
\item[] $(PCS):$ If $x\in E$ and $y \in E, \: y^* \in G(y),$ then
$$
f(x)=f(y)+\langle y^*, x-y \rangle \implies y^* \in  G(x).
$$
\end{theorem}

However, in the case that $E$ is unbounded (and even if $E$ is closed), condition $(PCS)$ cannot replace $(SCS)$, as the following example shows.
\begin{ex}
{\em Let $g(x, y)=|x|+\theta(x), \: (x,y) \in \R^2,$ where $\theta: \R \to \R$ is a symmetric convex function with $\theta(t)=t^2$ for $t\in [-1,1]$, and $\theta$ affine on  $[1,+\infty).$ Let $E=\{(0,0)\}\cup \{(x, y)\in\R^2 : |x|\geq \min\{1, e^{y}\}\}$, and define $f$ and $G$ on $E$ by
$$
f=g \quad \text{on} \quad E,
$$

$$
G(x,y)=\begin{cases}
\{\nabla g(x,y)\} & \textrm{ if } (x,y) \in E \setminus \lbrace (0,0) \rbrace \\
\{(0,0)\} & \textrm{ if } (x,y)=(0,0).
\end{cases}
$$
Then $G(E)$ is bounded, $G$ is upper semicontinuous and it can be checked that $(f,G)$ satisfies conditions $(C)$ and $(PCS)$ on $E$, but there is no convex function $F:\R^2\to\R$ such that $F=f$ and $\partial F=G$ on $E$, because for every convex function $\varphi:\R^2\to\R$ such that $\varphi=f$ on $E$ we must have $\varphi(x,y)=|x|+\theta(x)$ on $\R^2$, and in particular $\partial \varphi(0,0)=[-1,1]\times\{0\}$. As a matter of fact, for every pair of convex functions $\psi:\R^2\to\R$ and $\eta:\R\to\R$, we have that if $\psi(x,y)=\eta(x)$ for all $(x,y)\in E$ then $\psi(x,y)=\eta(x)$ for all $(x,y)\in\R^2$. This is an easy consequence of \cite[Theorem 1.11]{AzagraMudarraGlobalGeometry}, but next we also provide a direct proof of this assertion for the reader's convenience. We first claim that for every $(x_0, y_0)\in \R^2$ and every $(a,b)\in\partial\psi(x_0, y_0)$ we have $b=0$. Indeed, by convexity we have 
$$
\psi(x,y)\geq \psi(x_0, y_0)+a(x-x_0)+b(y-y_0)
$$
for all $(x,y)\in\R^2$. Taking $(x, y)$ of the form $(x(t), y(t))=(2, t)$, $t\in\R$, and noting that $(2,t)\in E$ for all $t\in\R$, we get
$$
\eta(2)=\psi(2,t)\geq \psi(x_0, y_0)+a(2-x_0)+b(t-y_0)
$$
for all $t\in\R$, which is impossible unless $b=0$. Thus we have that $\partial\psi(x,y)\subset\R\times\{0\}$ for all $(x,y)\in\R^2$, which implies that, for each $x\in\R$, the function $\R\ni y\mapsto \psi(x,y)\in\R$ does not depend on $y$. Since for every $(x,y)\in \R^2$ there exists some $y_0$ with $(x,y_0)\in E$ we deduce that $\psi(x,y)=\psi(x,y_0)=\eta(x)$. Thus $\psi(x,y)=\eta(x)$ for all $(x,y)\in\R^2$.

\medskip

However, note that if we set $G(0,0)=[-1, 1]\times\{0\}$ (instead of just $\{ (0,0)\}$) then the set-valued jet $(f, G)$ does satisfy all the hypotheses of Theorem \ref{maintheoremfinitedimensional}, and therefore there exists a convex function $\psi:\R^2\to\R$ such that $(\psi, \partial\psi)=(f, G)$ on $E$ (of course this function must be $\psi(x,y)=|x|+\theta(x)$).     }
\end{ex}

\medskip

Theorem \ref{maintheoremfinitedimensional} actually is a corollary of the following more general result for functions that are bounded on bounded subsets of a separable Banach space.

\begin{theorem}\label{maintheoremseparable}
Let $E$ be an arbitrary subset of a separable Banach space $X$, let $f: E \to \R$ be a function and let $G: E \rightrightarrows X^*$ be a set-valued mapping such that $G(x)$ is a nonempty convex $w^*$-compact subset of $X^{*}$ for every $x\in E.$ There exists a convex function $F: X \to \R$, bounded on bounded subsets and such that $F(x)=f(x)$ and $\partial F(x)= G(x)$ for every $x\in E$, if and only if the following conditions are satisfied.
\item[] $(C):$ $f(x) \geq f(y)+ y^* (x-y)$ for every $x,y\in E, \:y^*\in G(y).$ 
\item[] $(EX):$ If $(y_k)_k \subset E, \: (y_k^*)_k \subset X^*$ are such that $y_k^* \in G(y_k)$ for every $k$ and $\lim_k \|y_k^*\|_* =+\infty,$ then
$$
\lim_k \frac{ y_k^*(y_k)-f(y_k)}{\|y_k^*\|_*}= +\infty.
$$
\item[] $(SCS):$ If $x\in E$ and $(y_k)_k \subset E, \: (y_k^*)_k \subset X^*$ are such that $y_k^* \in G(y_k)$ for every $k$ and $(y_k^*)_k$ is bounded, then 
$$
\lim_k \left( f(x)-f(y_k)-y_k^* (x-y_k ) \right)=0 \implies  \exists \: (x_k^*)_k \subset G(x) \quad \text{such that} \quad w^*\text{-}\lim_k (y_k^*-x_k^*)=0.
$$ 
\end{theorem}

Without any restriction on the Banach space $X$, but still focusing on the special class of convex functions which are bounded on bounded subsets of $X$, we have the following result.

\begin{theorem}\label{maingeneraltheorem}
Let $X$ be a Banach space, let $E \subset X$ be arbitrary, $f: E \to \R$ be a function and $G: E \rightrightarrows X^*$ be a set-valued mapping such that $G(x)$ is a nonempty convex $w^*$-compact subset of $X^{*}$ for every $x\in E.$ There exists a convex function $F: X \to \R$ that is bounded on bounded subsets and such that $F(x)=f(x)$ and $\partial F(x)= G(x)$ for every $x\in E$ if and only if the following conditions are satisfied.
\item[] $(C):$ $f(x) \geq f(y)+ y^* (x-y)$ for every $x,y\in E, \:y^*\in G(y).$ 
\item[] $(EX):$ If $(y_k)_k \subset E, \: (y_k^*)_k \subset X^*$ are such that $y_k^* \in G(y_k)$ for every $k$ and $\lim_k \|y_k^*\|_* =+\infty,$ then
$$
\lim_k \frac{ y_k^*(y_k)-f(y_k)}{\|y_k^*\|_*}= +\infty.
$$
\item[] $(CS):$ If $x\in E$ and $(y_k)_k \subset E, \: (y_k^*)_k \subset X^*$ are such that $y_k^* \in G(y_k)$ for every $k$ and $(y_k^*)_k$ is bounded, then 
$$
\lim_k \left( f(x)-f(y_k)-y_k^* (x-y_k ) \right)=0 \implies G(x) \cap \overline{\lbrace y_k^* \rbrace_k}^{w^*} \neq \emptyset.
$$

\end{theorem}

Finally, a complete solution to our problem for functions that are not necessarily bounded on bounded sets is given in the following theorem.

\begin{theorem}\label{mainmoregeneraltheorem}
Let $X$ be a Banach space, let $E \subset X$ be arbitrary, $f: E \to \R$ be a function and $G: E \rightrightarrows X^*$ be a set-valued mapping such that $G(x)$ is a nonempty convex $w^*$-compact subset of $X^{*}$ for every $x\in E.$ There exists a continuous convex function $F: X \to \R$ such that $F(x)=f(x)$ and $\partial F(x)= G(x)$ for every $x\in E$ if and only if the following conditions are satisfied.
\item[] $(C):$ $f(x) \geq f(y)+ y^* (x-y)$ for every $x,y\in E, \:y^*\in G(y).$ 
\item[] $(GEX):$ If $(y_k)_k \subset E, \: (y_k^*)_k \subset X^*$ are such that $y_k^* \in G(y_k)$ for every $k$ and $\lim_k \|y_k^*\|_* =+\infty,$ then
$$
\lim_k \left( y_k^*(y_k-x)-f(y_k) \right)= +\infty \quad \text{for every} \quad x\in X.
$$
\item[] $(CS):$ If $x\in E$ and $(y_k)_k \subset E, \: (y_k^*)_k \subset X^*$ are such that $y_k^* \in G(y_k)$ for every $k$ and $(y_k^*)_k$ is bounded, then 
$$
\lim_k \left( f(x)-f(y_k)-y_k^* (x-y_k ) \right)=0 \implies G(x) \cap \overline{\lbrace y_k^* \rbrace_k}^{w^*} \neq \emptyset.
$$

\end{theorem}

\medskip

In fact, we will see that in the preceding theorems, an extension $F$ is given by the formula
$$
F(x)= \sup \lbrace f(y)+ y^*(x-y) \: : \: y\in E, \: y^*\in G(y) \rbrace,\quad x\in X.
$$
This extension has the property that, for every continuous convex function $H$ on $X$ such that $H=f$ and $\partial H \supset G$ on $E,$ we have $F \leq H$ on $X.$ Therefore $F$ is the minimal continuous convex extension of the datum $(f,G):E \to \R \times 2^{X^*}.$ 

\medskip

In order to help the reader get acquainted more quickly with the conditions of the above theorems, let us say a few words as to why we chose to label them $(C)$, $(EX)$, etc. Condition $(C)$ refers to {\em convexity} of the set-valued $1$-jet, meaning that all the function data lie above each of the putative supporting hyperplanes arising from the subdifferential data. Condition $(EX)$ and $(GEX)$ are related to the {\em existence} of at least one continuous convex function $F:X\to\R$ satisfying $F(x)=f(x)$ and $G(x)\subseteq \partial F(x)$ for all $x\in E$. In finite dimensions we only need the existence condition $(EX)$, but in the infinite-dimensional setting a {\em generalized existence condition} such as $(GEX)$ becomes necessary. These conditions are not sufficient to ensure the equality $G(x)=\partial F(x)$ for all $x\in E$, and that is why we need to introduce other conditions like $(CS)$, $(SCS)$, $(PCS)$,  which refer to how {\em convexity} forces {\em subdifferentials} to behave, both asymptotically and pointwise. For instance, $(PCS)$ tells us that in order that a convex extension $F$ with prescribed subdifferential $\partial F=G$ exists, the given data must satisfy the property that if a point $(x, f(x))$ of the graph of $f$ is touched by a putative supporting hyperplane at some other point $y$, then that hyperplane must also be one of the putative hyperplanes of $F$ at $x$. When $E$ is not compact or $G$ is not upper semicontinuous this easy {\em pointwise} condition is no longer sufficient, and even in the finite-dimensional situation we must consider sequences instead of points, as in condition $(SCS)$, which tells us that if a sequence of putative hyperplanes at points $(y_k)$ asymptotically touches a point $(x,f(x))$ in the graph of $f$ then there is a corresponding sequence of putative supporting hyperplanes at $x$ with a similar asymptotic behavior. In the nonseparable case the situation becomes even more complicated, and we need condition $(CS)$ in place of the weaker condition $(SCS)$. 

\medskip

Let us finish this introduction by examining three variations of an example of Borwein, Montesinos and Vanderwerff \cite[Example 4.2]{BorweinMontesinosVanderwerff} in the light of Theorem \ref{mainmoregeneraltheorem}. We wish to use these examples to illustrate the role of conditions $(GEX)$, $(EX)$ and $(CS)$ in an infinite-dimensional setting.

\begin{example}{\em $(1)$ Let $E=B_{\ell_2}$ be the open unit ball of $\ell ^2$. The function $f:E\to \mathbb{R}$ defined by 
$$
f(x)=\sum_{n=1}^{\infty}x_n^{2n}, \quad x=(x_n)_{n\geq 1} \in E,
$$
is convex, bounded and differentiable (in fact real-analytic). The same formula defines a real-analytic convex function on $c_0$, so in particular $f$ can be extended to a continuous convex function $F$ on $c_0$ with $\partial F = \partial f$ on $E$, and hence $(f, \partial f)$ satisfies condition $(GEX)$ in Theorem \ref{mainmoregeneraltheorem} with $X=c_0$. However $f$ has no continuous convex extension from $B_{\ell^2}$ to all of $\ell ^{\infty}$. Indeed, let us assume that there exists a continuous convex function $F$ on $\ell^\infty$ with $F=f$ on $E.$ Then $(f, \partial F)$ must satisfy condition $(GEX)$ of Theorem \ref{mainmoregeneraltheorem} on the set $E$ with $X=\ell^\infty.$ For every $k \in \N $, define $y_k:=r_ke_k$, where $r_k$ is such that $\frac{3}{4}=r_k^{2k-1}$, and let $y^*_k\in \partial F(y_k)\subset\ell^{\infty}$. Then $y_k \in E$ and, since $F=f$ on $E,$ it is clear that the linear functionals $y^*_k$ and $Df(y_k)=2kr_k^{2k-1}e_k^*$ coincide on $\ell^2$. Moreover, by the density of $\ell^2$ in $c_0$ and the continuity of $y_k^{*}$, $e_k^{*}$ with respect to the norm $\|\cdot\|_{\infty}$, the linear functionals $y^*_k$ and $2kr_k^{2k-1}e_k^*$ must also agree on $c_0$. That is, we have
\begin{equation}\label{functionalsareequalc0}
y_k^*=2kr_k^{2k-1}e_k^* = \tfrac{3}{2} k e_k^* \quad \text{on} \: \: c_0, \quad \text{for every} \quad k\in \N.
\end{equation}
Because $F$ is continuous at $0$, there exists $\delta>0$ such that $F(x)\leq \frac{3}{4}$ whenever $\|x\|_{\infty}\leq \delta$. Then
$$
\tfrac{3}{4}\geq F(x)\geq f(y_k)+y_k^*(x-y_k)=r_k^{2k}+y^*_k(x)-2kr_k^{2k} \quad \text{whenever} \quad \|x\|_{\infty}\leq \delta.
$$
This implies that 
$$
\|y_k^{*}\|_*\leq \delta^{-1} \left( \tfrac{3}{4} +(2k-1) r_k^{2k} \right) =  \tfrac{3}{4} \delta^{-1} \left( 1 +(2k-1) r_k \right) \leq \frac{3k}{2\delta} \quad \text{for every} \quad k\in \N.
$$
The preceding inequality together with \eqref{functionalsareequalc0} yield
\begin{equation}\label{estimationnorm}
\tfrac{3}{2}k \leq  \|y_k^{*}\|_*\leq \tfrac{3k}{2\delta} \quad \text{for every} \quad k\in \N. 
\end{equation}
If for every $x\in \ell^\infty,$ the sequence $\left( \frac{2}{3}k^{-1} y_k^*(x) \right)_k$ converges to $0,$ then, using \eqref{functionalsareequalc0} and \eqref{estimationnorm}, we deduce that the mapping $ \ell^\infty \ni x \mapsto  \left( \frac{2}{3}k^{-1} y_k^*(x) \right)_k$ defines a bounded linear projection onto $c_0,$ which is absurd since $c_0$ is not complemented in $\ell^\infty.$ Thus there exist $x\in \ell ^\infty$ and a subsequence $(k_n)_n$ 
such that $ \frac{2}{3}k_n^{-1}y_{k_n}^{*}(x)\geq 1$ for all $n$. Bearing in mind the first inequality in \eqref{estimationnorm} we have
$$
     y_{k_n}^*(y_{k_n}-x)-f(y_{k_n})\leq 2k_nr_{k_n}^{2k_n}-\tfrac{3}{2}k_n-r_{k_n}^{2k_n}=
     \tfrac{3}{2}k_nr_{k_n}-\tfrac{3}{2}k_n-\tfrac{3}{4}r_{k_n}
     \leq 0 \quad \text{for every} \quad n.
$$
This shows that condition $(GEX)$ of Theorem \ref{mainmoregeneraltheorem} with $X=\ell ^{\infty}$ fails for $(f, \partial F)$ on $E,$ a contradiction. We conclude that there is no continuous convex extension of $f$ to $\ell^\infty.$ 

\medskip
$(2)$ If $E$ and $f$ are as in $(1),$ there is no convex function $F:\ell^2 \to \R$ which is bounded on bounded subsets and such that $F=f$ on $E.$ Indeed, let $F$ be a continuous convex function $F: \ell^2 \to \R$ with $F=f$ on $E.$ We learnt from $(1)$ that $F$ must satisfy $y_k^*= \frac{3}{2}ke_k^*$ on $\ell_2$ for every $y_k^*\in \partial F(y_k)$ and every $k.$ We have that $\lim_k \|y_k^*\|_*= \infty$ while
$$
\lim_k \frac{y_k^*(y_k)-f(y_k)}{\|y_k^*\|_*}= \lim_k \frac{2kr_k^{2k}-r_k^{2k}}{2k r_k^{2k-1}} = \lim_k \frac{2k-1}{2k} r_k = 1.
$$
This shows that condition $(EX)$ of Theorem \ref{maingeneraltheorem} is not fulfilled for $(F, \partial F)$ and therefore $F$ is not bounded on bounded subsets.

\medskip

$(3)$ Let $r\in (0,1)$ and set $E=r B_{\ell_2},$ $X=\ell^\infty$. Define $f:E \to \R,$ $G:E \rightrightarrows X^*$ by
$$
f(x)= \sum_{n=1}^\infty x_n^{2n}, \quad G(x)=\left\lbrace g(x):=\sum_{n=1}^\infty 2n x_n^{2n-1} e_n^* \right\rbrace, \quad x=(x_n)_{n\geq 1} \in E.
$$
Then $f$ admits a Lipschitz convex extension $F$ to all of $X$ such that $\partial F= G$ on $E.$ Indeed, an easy calculation shows that $\| g(x)\|_* \leq \sum_{n = 1}^\infty r^{2n-1} = 2r (1-r^2)^{-2}$ for every $x\in E.$ Besides, $(f,G)$ satisfies condition $(C)$ on $E$ and thus the formula
$$
F(x)= \sup_{y\in E} \lbrace f(y)+g(y)(x-y) \rbrace, \quad x\in X,
$$
defines a Lipschitz convex function with $F=f$ on $E$ and $g(x) \in \partial F(x)$ for every $x\in E.$ In order to see that $\partial F= G$ on $E,$ let us check that $(f,G)$ satisfies condition $(CS)$ of Theorem \ref{maingeneraltheorem}. Assume that $x=(x_n)_n \in E$, $ \left( y_k=(y_n^k)_n \right)_k\subset  E$, and $y_k^*=g(y_k)$ are such that
$$
\lim_k\sum_{n=1}^{\infty} \left( x_n^{2n}-(y^k_n)^{2n}-2n(y^k_n)^{2n-1}(x_n-y^k_n) \right)=\lim_k \left( f(x)-f(y_k)-g(y_k)(x-y_k) \right)=0.
$$
This implies that
$$
\lim_k \left( \varphi_n(x_n)-\varphi_n(y_n^k)-\varphi_n'(y_n^k)(x_n-y_n^k) \right)=0, \quad \text{where} \quad \varphi_n(t)=t^{2n}, \quad t\in \R, \: n\in \N.
$$
Since each $\varphi_n : \R \to \R$ is strictly convex we must have $\lim_k y_n^k=x_n$ for every $n\in \N.$ Now observe that 
$$
2n \, \big | (y^k_n)^{2n-1}-x_n^{2n-1} \big | \leq 4nr^{2n-1} \quad \text{for every} \quad k\in \N, \quad \text{where} \quad \sum_{n=1}^\infty 4nr^{2n-1}<\infty.
$$ 
Thus, for every $u\in X,$ the Dominated Convergence Theorem gives us
$$
\lim_k |(y_k^*-g(x))(u)| \leq  \lim_k \left( \sum_{n=1}^{\infty} 2n \, \big |(y^k_n)^{2n-1}-x_n^{2n-1} \big | \right) \|u\|_\infty=\|u\|_\infty \sum_{n=1}^{\infty}\lim_k 2n \, \big |(y^k_n)^{2n-1}-x_n^{2n-1} \big |=0,
$$
which shows that $g(x) \in \overline{\lbrace g(y_k) \rbrace_k}^{w^*}.$ Thus condition $(CS)$ of Theorem \ref{maingeneraltheorem} is satisfied for $(f,G)$ on $E.$ 
}
\end{example}

\section{Proofs of theorems \ref{mainmoregeneraltheorem}, \ref{maingeneraltheorem} and \ref{maintheoremseparable}}\label{sectiongeneraltheorem}

Throughout this section $\big ( X,\| \cdot \| \big )$ will be a Banach space and $\|\cdot \|_*$ will denote the dual norm on $X^*.$ We will start with the proof of the most general results of the paper, that is, Theorems \ref{mainmoregeneraltheorem} and \ref{maingeneraltheorem}. Then, a small observation (see Remark \ref{remarkcsseparablescs} below) will allow us to deduce Theorem \ref{maintheoremseparable} for separable spaces.

\subsection{Theorems \ref{mainmoregeneraltheorem} and \ref{maingeneraltheorem}. Only if part.} Let $F: X \to \R$ be a continuous convex function. We obviously have $F(x) \geq F(y)+  y^* (x-y)$ for every $x,y \in X, \: y^* \in \partial F(y).$ Thus condition $(C)$ is necessary in Theorems \ref{mainmoregeneraltheorem} and \ref{maingeneraltheorem}. 

\medskip

Let us now prove that $(GEX)$ is a necessary condition in Theorem \ref{mainmoregeneraltheorem}. Let $x\in X$ and let $(y_k)_k \subset X, \: (y_k^*)_k \subset X^*$ with $y_k^* \in \partial F(y_k)$ and $\lim_k \| y_k^*\|_* = +\infty.$ Assume, for the sake of contradiction, that $\lim_k (y_k^*(y_k-x)-F(y_k)) \neq + \infty.$ Then, after passing to a subsequence, we may find $M>\max\lbrace F(x), 0 \rbrace$ such that
\begin{equation}\label{inequalitynecessitycondition}
y_k^*(y_k-x) - F(y_k) \leq M \quad \text{for every} \quad k.
\end{equation}
Now we consider, for every $k,$ a vector $v_k \in X$ with
\begin{equation}\label{propertiesvectorvknecessity}
\| v_k\| \leq 1 \quad \text{and} \quad y_k^*(v_k) \geq \frac{1}{2}\|y_k^*\|_*.
\end{equation}
If we define $z_k = x+\frac{4M}{\| y_k^*\|_*}v_k,$ then we get
$$
0 \leq F(z_k)-F(y_k)+ y_k^*( y_k- z_k) \leq F(z_k)+ M  - \frac{4M}{\|y_k^*\|_*}y_k^*(v_k) \leq F(z_k)-M,
$$
where the first inequality follows from the convexity of $F,$ the second one from \eqref{inequalitynecessitycondition}, and the third one from \eqref{propertiesvectorvknecessity}. Since $(z_k)_k$ converges to $x,$ the continuity of $F$ at $x$ gives that $M \leq \lim_k F(z_k) = F(x),$ contradicting the choice of $M.$

\medskip

In order to see that $(EX)$ is necessary in Theorem \ref{maingeneraltheorem}, assume further that $F$ is bounded on bounded subsets. Suppose that $(y_k^*)_k $ is a sequence such that $\lim_k \|y_k^*\|_* = +\infty$ and $y_k^* \in \partial F(y_k)$ for every $k$ but there exists $M>0$ with
\begin{equation}\label{inequalitynecessityconditionbboundedsets}
 y_k^*(y_k)-F(y_k) \leq M \|y_k^*\|_* \quad \text{for every} \quad k.
\end{equation}
Consider, for every $k,$ a vector $v_k \in X$ with 
\begin{equation}\label{propertiesvectorvknecessitybboundedsets}
\| v_k \| \leq 1 \quad \text{and} \quad y_k^*(v_k) \geq \frac{1}{2}\| y_k^*\|_*.
\end{equation} 
If we define $z_k= 4M v_k,$ we can write
\begin{align*}
0 & \leq F(z_k)-F(y_k)-y_k^* (z_k-y_k) =F(z_k)+ \big( y_k^*(y_k)-F(y_k) \big)-y_k^*(z_k) \\
& \leq  F(z_k)+M\|y_k^*\|_* -2 M \|y_k^*\|_*= F(z_k)-M\|y_k^*\|_*,
\end{align*}
where the first inequality follows from the fact that $y_k^* \in \partial F(y_k),$ and the second one from \eqref{inequalitynecessityconditionbboundedsets} and \eqref{propertiesvectorvknecessitybboundedsets}. Since $(F(z_k))_k$ is a bounded sequence and $\lim_k \| y_k^*\|= \infty,$ the last chain of inequalities yields a contradiction. This proves that $\lim_k \frac{y_k^*(y_k)-f(y_k)}{\|y_k^*\|_*}= +\infty.$

\medskip

Finally, in order to show that $(CS)$ is a necessary condition in Theorems \ref{mainmoregeneraltheorem} and \ref{maingeneraltheorem}, let us first prove the following fact.

\begin{fact}\label{factCScompact}
{\em Let $h: X \to \R$ be a continuous convex function and let $x,y\in X$ be two points such that $h(x)=h(y)+y^*(x-y)$ for some $y^*\in \partial h(y).$ Then $y^* \in \partial h(x).$ }
\end{fact}  
\begin{proof}
Because $y^*\in \partial h(y),$ we can write, for every $z \in X$,
$$
h(z) \geq h(y)+ y^*(z-y) = h(y)+y^*(x-y) + y^*(z-x) = h(x) + y^*(z-x),
$$
that is $y^* \in \partial h(x). $
\end{proof}

Now assume, seeking a contradiction, that $(CS)$ is not satisfied. Then we can find $ x\in X$ and sequences $(y_k)_k, \: (y_k^*)_k$ such that $y_k^* \in  \partial F(y_k), \: (y_k^*)_k$ is bounded and
\begin{equation}\label{limitequal0necessity}
\lim_k \left( F(x)-F(y_k)- y_k^* (x-y_k) \right)=0,
\end{equation}
but $ \partial F(x) \cap \overline{\lbrace y_k^* \rbrace_k}^{w^*} =\emptyset.$ Since $( y_k^* )_k$ is bounded, we can find a subnet $( y_{k_\alpha}^* )_{\alpha \in D}$ of $( y_k^* )_k$ $w^{*}$-convergent to $\xi \in X^*.$ Obviously, $\xi$ does not belong to $\partial F(x)$ and, since $\partial F(x)$ is $w^{*}$-closed, we can apply the Hahn-Banach Theorem for $(X^*,w^*)$ in order to find $v\in X$ with 
\begin{equation}\label{separatingv}
\xi (v) > \sup_{x^* \in \partial F(x)}  x^* (v).
\end{equation}
For every $\alpha \in D,$ the number $r_\alpha = F(x)-F(y_{k_\alpha})-y_{k_\alpha}^*(x-y_{k_\alpha})$ is strictly positive (as otherwise, $y_{k_\alpha}^* \in \partial F(x)$ by Fact \ref{factCScompact}) and $ \lim_\alpha r_\alpha = 0$ by \eqref{limitequal0necessity}. We now write
\begin{align}\label{estimationbiggernecessity}
F( x+ \sqrt{r_\alpha} v ) -F(x) & \geq F(y_{k_\alpha})+ y_{k_\alpha}^* (x+ \sqrt{r_\alpha} v-y_{k_\alpha}) -F(x) = -r_\alpha +y_{k_\alpha}^* (v) \sqrt{r_\alpha}   \\ 
& = -r_\alpha + (y_{k_\alpha}^*-\xi) (v) \sqrt{r_\alpha} + \xi (v) \sqrt{r_\alpha}, \nonumber
\end{align}
for every $\alpha \in D.$ Let us consider a net $(z_\alpha^*)_{\alpha \in D}$ such that $z_\alpha^* \in \partial F (x+ \sqrt{r_\alpha} v )$ for each $\alpha.$ Since the net $\lbrace x+ \sqrt{r_\alpha} v \rbrace_{\alpha \in D}$ strongly converges to $x,$ for any $\varepsilon>0,$ the $\|\cdot\|$-$w^{*}$-upper semicontinuity of $\partial F$ (see \cite[Proposition 6.1.1]{BorweinVanderwerffbook}) gives $\alpha_\varepsilon \in D$ such that for every $\alpha\in D$ with $\alpha_\varepsilon \leq \alpha,$ we can find $x^*_{\varepsilon,\alpha} \in \partial F(x)$ with $| (z_\alpha^*-x^*_{\varepsilon,\alpha})(v)| \leq \varepsilon.$ Let $\varepsilon >0$ and $\alpha \in D$ with $\alpha_\varepsilon \leq \alpha.$ From the convexity of $F$ and the fact that $z_\alpha^* \in \partial F( x+ \sqrt{r_\alpha} v ),$ it follows that
$$
F( x+ \sqrt{r_\alpha} v -\sqrt{r_\alpha} v ) - F(x+ \sqrt{r_\alpha} v) \geq z_\alpha^* ( - \sqrt{r_\alpha} v),
$$
and so
\begin{equation}\label{estimationsmallernecessity}
F( x+ \sqrt{r_\alpha} v ) -F(x) \leq z_\alpha^* (v)  \sqrt{r_\alpha}  \leq \Big ( (z_\alpha^*-x^*_{\varepsilon,\alpha})(v) + \sup_{x^* \in \partial F(x)}  x^*(v) \Big ) \sqrt{r_\alpha} \leq \Big ( \varepsilon + \sup_{x^* \in \partial F(x)}  x^*(v) \Big ) \sqrt{r_\alpha}
\end{equation}
Dividing by $\sqrt{r_\alpha}$ in \eqref{estimationbiggernecessity} and \eqref{estimationsmallernecessity}, we obtain
$$
 \varepsilon + \sup_{x^* \in \partial F(x)}  x^*(v)  \geq -\sqrt{r_\alpha} + (y_{k_\alpha}^*-\xi)(v)  + \xi(v) \quad \text{for all} \quad \alpha \in D,\: \alpha_\varepsilon \leq \alpha.
$$
Because the net $\lbrace -\sqrt{r_\alpha} + (y_{k_\alpha}^*-\xi)(v)\rbrace_{\alpha\in D}$ converges to $0,$ we obtain from the preceding inequality that $\xi(v) \leq \varepsilon + \sup_{x^* \in \partial F(x)}  x^*(v).$ Since $\varepsilon >0$ is arbitrary, this contradicts \eqref{separatingv}.

\subsection{Theorems \ref{mainmoregeneraltheorem} and \ref{maingeneraltheorem}. If part.} Let us assume that $(f,G)$ satisfies conditions $(C), (GEX)$ and $(CS)$ of Theorem \ref{mainmoregeneraltheorem} and define
$$
F(x) = \sup_{y\in E} \sup_{y^* \in G(y)} \lbrace f(y) +y^* (x-y) \rbrace, \quad x\in X.
$$

\begin{claim}\label{claimfiniteness}
{\em $F$ is finite everywhere in $X.$ }
\end{claim}
\begin{proof}
Consider $x \in X$ and sequences $(y_k)_k \subset E, \:(y_k^*)_k $ with $y_k^* \in G(y_k)$ for every $k,$ such that 
\begin{equation}\label{sequenceykyk*0}
F(x) = \lim_k \left( f(y_k) + y_k^* (x-y_k) \right). 
\end{equation}
If we take some $z_0 \in E,$ condition $(C)$ yields $f(z_0) \geq f(y_k)+  y_k^*(z_0-y_k),$ hence
$$
f(y_k) + y_k^* (x-y_k) \leq f(z_0) + y_k^* (x-z_0).
$$
This shows that $F(x)$ will be finite as soon as we prove that $(y_k^*)_k$ is bounded. Assume, for the sake of contradiction, that $(y_k^*)_k$ is unbounded. Then condition $(GEX)$ tells us that, possibly after passing to a subsequence, $\lim_k \left( f(y_k)+y_k^*(x-y_k) \right)= -\infty.$ This contradicts \eqref{sequenceykyk*0}, since obviously $F(x) > -\infty.$  
\end{proof}

\begin{claim}\label{claimfinitenessboundedbounded}
{\em Assuming further that $(f,G)$ satisfies condition $(EX)$ of Theorem \ref{maingeneraltheorem}, $F$ is bounded on bounded subsets of $X.$ }
\end{claim}
\begin{proof}
Since condition $(EX)$ is stronger that $(GEX)$ of Theorem \ref{mainmoregeneraltheorem}, we already know that $F$ is finite everywhere in this case, by virtue of Claim \ref{claimfiniteness}. Let us fix $z_0 \in E$ and $z_0^* \in G(z_0).$ Assume, seeking a contradiction, that $B$ is a bounded subset of $X$ for which $F|_B$ is unbounded. Then we can find $(x_k)_k \subset B$ such that $\lim_k F(x_k)=+\infty.$ By the definition of $F(x_k),$ we can find sequences $(y_k)_k \subset E, \:(y_k^*)_k $ with $y_k^* \in G(y_k)$ for every $k,$ such that 
\begin{equation}\label{sequenceykyk*1}
F(x_k) \leq f(y_k) + y_k^* (x_k-y_k) + \frac{1}{k} \quad \text{for every } \: k. 
\end{equation}
Moreover, since $f(z_0) + z_0^* (x_k-z_0 )$ is one of the expressions considered in the definition of $F(x_k),$ the sequences $(y_k)_k,$ $(y_k^*)_k$ can be even selected so that
\begin{equation}\label{sequenceykyk*2}
f(z_0) + z_0^* (x_k-z_0 ) \leq f(y_k) + y_k^* (x_k-y_k)  \quad \text{for every } \: k. 
\end{equation}
Condition $(C)$ implies that
$$
f(y_k) + y_k^* (x_k-y_k) \leq f(z_0) + y_k^* (y_k-z_0)  + y_k^* (x_k-y_k ) = f(z_0) + y_k^* (x_k-z_0).
$$
Since $(x_k)_k$ is bounded, the preceding inequality shows that $(y_k^*)_k$ must be unbounded, as otherwise \eqref{sequenceykyk*1} would give that $(F(x_k))_k$ is bounded, a contradiction. Passing to a subsequence we may assume that $\lim_k \|y_k^*\|_*= +\infty,$ and then condition $(EX)$ tells us that $ \lim_k \frac{y_k^*(y_k)-f(y_k)}{\|y_k^*\|_*}= +\infty.$ Using \eqref{sequenceykyk*2} we easily obtain
$$
\frac{y_k^*(y_k)-f(y_k)}{\|y_k^*\|_*} \leq  \frac{1}{\|y_k^*\|_{*}} y_k^*  (x_k) - \frac{f(z_0) + z_0^*  (x_k-z_0)}{\|y_k^*\|_*},
$$
where the last term is bounded above. This yields a contradiction and shows that $F$ is bounded on $B.$ 
\end{proof}

\begin{claim}
{\em $F$ is continuous and convex on $X,$ $F=f$ and $G \subset \partial F$ on $E.$ }
\end{claim}
\begin{proof}
The function $F,$ being the supremum of a family of lower semicontinuous convex functions, is convex and lower semicontinuous as well. Moreover, we learnt from Claim \ref{claimfiniteness} that $\dom(F) =X.$ Since $X$ is a Banach space, every lower semicontinuous convex function on $X$ is continuous on $\interior(\dom(F))$ (see for instance \cite[Proposition 4.1.5, p. 129]{BorweinVanderwerffbook}), hence $F$ is continuous on $X.$

\medskip

The inequality $F \geq f$ on $E$ is obvious by definition of $F$, and the converse inequality follows immediately from condition $(C).$ Finally, for every $x\in E, z\in X, x^* \in G(x)$, the definition of $F$ and the equality $F=f$ on $E$ give
$$
F(z) \geq f(x) + x^* (z-x) = F(x)+ x^* ( z-x ),
$$
and then $x^* \in \partial F(x).$ 
\end{proof}

To conclude the proof of Theorems \ref{mainmoregeneraltheorem} and \ref{maingeneraltheorem} it only remains to prove the following.
\begin{claim}\label{claimsamesubdifferential}
{\em $\partial F= G$ on $E.$}
\end{claim}
\begin{proof}
Let $x\in E$ and suppose that there exists $\xi \in \partial F(x) \setminus G(x).$ Since $G(x)$ is $w^*$-closed and convex, the Hahn-Banach Theorem for $(X^*,w^*)$ provides us with some $u\in X$ such that
\begin{equation}\label{separatingu}
\xi (u) > \sup_{x^*\in G(x)} x^*(u).
\end{equation}
We now pick two sequences $(y_k)_k \subset E, \: (y_k^*)_k \subset X$ with $y_k^* \in G(y_k)$ such that
\begin{equation}\label{sequenceykyk*3}
F(x+\tfrac{1}{k}u ) \geq f(y_k)+ y_k^* (x+ \tfrac{1}{k} u-y_k) \geq F(x+\tfrac{1}{k}u ) - \tfrac{1}{2^k} \quad \text{for every } \: k.
\end{equation}
The sequence $(y_k^*)_k$ must be bounded. Indeed, let us assume that $\lim_k \| y_k^*\|_* = +\infty.$ By condition $(GEX)$ we have that
\begin{equation}\label{limitx+u}
\lim_k \left( f(y_k)+y_k^*(x+u-y_k) \right)=-\infty.
\end{equation}
On the other hand, using the convexity of $F$ in combination with \eqref{sequenceykyk*3} and taking some $x^*\in G(x),$ we have that
\begin{align*}
f(y_k) + y_k^* (x+ \tfrac{1}{k} u-y_k)& \geq F(x+\tfrac{1}{k}u ) - 2^{-k} \geq f(x) +x^* \left( x+\tfrac{1}{k}u -x \right) -2^{-k} \\
&   \geq f(y_k) + y_k^* (x-y_k) + x^*( \tfrac{1}{k} u )-2^{-k},
\end{align*}
which implies that 
$$
y_k^*(u) \geq x^*(u)-k 2^{-k}
$$ 
for every $k.$ This shows that $( y_k^*(u))_k$ is bounded below and then, by virtue of \eqref{sequenceykyk*3} and \eqref{limitx+u}, we obtain
$$
F(x)=\lim_k \left( f(y_k)+y_k^*(x+\tfrac{1}{k} u-y_k) \right) = \lim_k \left( f(y_k)+ y_k^*(x+u-y_k)- (1- \tfrac{1}{k}) y_k^*(u) \right) = -\infty,
$$
which is absurd. Thus $(y_k^*)_k$ is a bounded sequence. Observe that the fact that $(y_k^*)_k$ is bounded together with \eqref{sequenceykyk*3} imply that 
$$
\lim_k \left( f(x)-f(y_k)- y_k^* (x-y_k ) \right)=0.
$$ 
Hence, condition $(CS)$ gives a point $x_0^* \in G(x) \cap \overline{\lbrace y_k^* \rbrace_k}^{w^*} .$ Given $\varepsilon >0,$ we can find $k=k_\varepsilon \in \N$ such that $| (y_k^*-x_0^*)(u)| \leq \varepsilon$ and $k 2^{-k} \leq \varepsilon.$  Using that $\xi \in \partial F(x)$ and \eqref{sequenceykyk*3} we can write
\begin{align*}
\tfrac{1}{k} \xi (u) & \leq F(x+\tfrac{1}{k}u)-F(x) = F(x+\tfrac{1}{k}u)-f(x) \leq F(x+\tfrac{1}{k}u) -f(y_k)-y_k^* (x-y_k) \\
& \leq \tfrac{1}{k} y_k^*( u )+ \tfrac{1}{2^k} = \tfrac{1}{k} (y_k^*-x_0^*)(u) + \tfrac{1}{k} x_0^* (u) + \tfrac{1}{2^k} \leq \tfrac{1}{k} (y_k^*-x_0^*)(u) + \tfrac{1}{k} \sup_{x^* \in G(x)}  x^*( u)  + \tfrac{1}{2^k}.
\end{align*}
We thus have that
$$
 \xi  (u) \leq (y_k^*-x_0^*)(u) + \sup_{x^* \in G(x)} x^* (u) + \tfrac{k}{2^k} \leq 2 \varepsilon + \sup_{x^* \in G(x)} x^* (u),
$$
and letting $\varepsilon \to 0^+$ we obtain $\xi (u) \leq \sup_{x^*\in G(x)} x^*(u) ,$ which contradicts \eqref{separatingu}.
\end{proof}

\medskip

\subsection{Proof of Theorem \ref{maintheoremseparable}} It is enough to apply Theorem \ref{maingeneraltheorem} in combination with the following remark.
\begin{remark}\label{remarkcsseparablescs}
{\em If $X$ is separable, condition $(CS)$ in Theorems \ref{maingeneraltheorem} and \ref{mainmoregeneraltheorem} is equivalent to condition $(SCS)$ in Theorem \ref{maintheoremseparable}.}
\end{remark}
\begin{proof}
Assume that $(f,G)$ satisties condition $(CS)$ on a subset $E$ and consider $x\in E$ and $(y_k)_k \subset E, \: (y_k^*)_k \subset X^*$ such that $y_k^* \in G(y_k)$ for every $k, \: (y_k^*)_k$ is bounded and 
$$
\lim_k \left( f(x)-f(y_k)-y_k^* (x-y_k ) \right)=0.
$$ 
Since $X$ is separable, the bounded subset $G(x) \cup \overline{\lbrace y_k^* \rbrace_k}^{w^*}$ of $(X^*, w^*)$ is metrizable; let us denote a suitable distance by $d.$ By the $w^{*}$-compactness of $G(x),$ we can find a sequence $(x_k^*)_k \subset G(x)$ such that 
\begin{equation}\label{comparingdistanceseparable}
d(y_k^*, x_k^*)  =\dist(y_{k}^{*}, G(x)) := \inf \lbrace d(y_k^*,x^*) \: : \: x^*\in G(x) \rbrace \quad \text{for every} \: \: k.
\end{equation}
Assume, for the sake of contradiction, that $d(y_k^*,x_k^*)$ does not tend to $0.$ Then we can find a subsequence $(k_j)_j,$ a positive $\varepsilon$ and $\xi \in X^*$ such that $(y_{k_j}^*)_j$ $w^{*}$-converges to $\xi$ and $d(y_{k_j}^*, x_{k_j}^*) \geq \varepsilon$ for every $j.$ Then condition $(CS)$ says that $G(x) \cap \lbrace \xi, y_{k_j}^* \rbrace_j  \neq \emptyset$, which, because $\dist(y_{k_j}^{*}, G(x))\geq \varepsilon$, implies $\xi \in G(x),$ hence $d(y_{k_j}^*, \xi)  \geq \varepsilon$ for every $j$ by \eqref{comparingdistanceseparable}. This contradicts the fact that $w^*$-$\lim_j y_{k_j}^*= \xi$ and therefore $w^*$-$\lim( y_k^*-x_k^*)=0.$ 
\end{proof}

\section*{Acknowledgements}
We wish to thank the referees for several suggestions that made us improve the exposition of the paper. The authors were partially supported by Grant MTM2015-65825-P. \, D. Azagra and C. Mudarra also acknowledge financial support from the Spanish Ministry of Economy and Competitiveness, through the ``Severo Ochoa for Centres of Excellence in R\&D'' (Grant SEV-2015-0554).

\end{document}